\documentclass[10pt, reqno]{amsart}

\usepackage{amsmath}
\usepackage{cite}
\usepackage{amssymb}
\usepackage{mathrsfs}  
\usepackage{hyperref}
\usepackage{enumitem}

\newtheorem{theorem}{Theorem}[section]
\newtheorem{lemma}[theorem]{Lemma}

\newtheorem{proposition}[theorem]{Proposition}
\newtheorem{remark}[theorem]{Remark}

\theoremstyle{definition}
\newtheorem{definition}[theorem]{Definition}

\newcommand\R{\mathbb{R}}

\newcommand\C{\mathbb{C}}

\newcommand{\F}{\mathcal{F}}
\newcommand{\M}{\mathcal{M}}
\newcommand{\D}{\mathcal{D}}

\newcommand{\qtq}[1]{\quad\text{#1}\quad}

\let\Im=\undefined\DeclareMathOperator{\Im}{Im}

\newcommand{\ol}[1]{\overline{#1}}
\newcommand{\wh}[1]{\widehat{#1}}

\newcommand{\norm}[1]{\|#1\|}
\newcommand{\snorm}[1]{\|#1\|}
\newcommand{\bnorm}[1]{\| #1\|}

\usepackage{color}

\numberwithin{equation}{section}

\allowdisplaybreaks

\begin{document}

\title[Decay and scattering for CCM]{Dispersive decay and scattering for \\ continuum Calogero--Moser models}

\author[R.~Killip]{Rowan Killip}
\address{CEREMADE, CNRS, Universit\'e Paris Dauphine--PSL, Place du Mar\'echal de Lattr\'e de Tassigny, 75016 Paris, France}
\email{killip@ceremade.dauphine.fr}

\author[T.~Laurens]{Thierry Laurens}
\address{Department of Mathematics, University of Michigan, Ann Arbor, MI 48109, USA}
\email{tlaurens@umich.edu}

\author[J.~Murphy]{Jason Murphy}
\address{Department of Mathematics, University of Oregon, Eugene, OR 97403, USA}
\email{jamu@uoregon.edu}

\author[M.~Vi\c{s}an]{Monica Vi\c{s}an}
\address{Institute of Science and Technology Austria (ISTA), Am Campus 1, 3400 Klosterneuburg, Austria}
\email{monica.visan@ist.ac.at}

\begin{abstract} We prove pointwise decay and scattering for small-mass solutions to the focusing and defocusing continuum Calogero--Moser models under suitable decay assumptions on the initial data. Our proof is based on an explicit formula for solutions and the exact conservation of the Galilean vector field associated to the equation.
\end{abstract}

\maketitle


\section{Introduction} 

The focusing and defocusing continuum Calogero--Moser models
\begin{equation}\label{CCM}\tag{CCM}
i \tfrac{d}{dt} q = - q'' \pm  2iqC_+\big(|q|^2\big)'
\end{equation}
comprise a pair of completely integrable PDEs that have received considerable interest in recent years.  In the equations above, primes denote spatial derivatives, upper/lower signs correspond to the focusing/defocusing cases, and the operator $C_+$ denotes the orthogonal projection from $L^2(\R)$ onto the Hardy space
\[
L^2_+(\R) := \{f\in L^2(\R):\, \widehat{f}(\xi) = 0 \qtq{for} \xi<0\}.
\]
For a discussion of the physical origins of \eqref{CCM}, we refer the reader to \cite{KLV:CCM} and the references therein (cf. \cite{Pelinovsky1995, Abanov2009, Calogero1, Calogero2, Moser}). 

As observed in \cite{Abanov2009}, \eqref{CCM} conserve the energies
\[
H_\pm(q) = \tfrac12 \int_\R |q'\mp iqC_+(|q|^2)|^2\,dx.
\]
The equations additionally conserve the mass
\[
M(q) = \int_\R |q|^2\,dx, 
\]
which is invariant under the scaling symmetry of \eqref{CCM}
\[
q(t,x)\mapsto q_\lambda(t,x) = \lambda^{-\frac12} q\bigl( \tfrac{t}{\lambda^2}, \tfrac{x}{\lambda}\bigr), \qquad \lambda>0.
\]
In particular, \eqref{CCM} are mass-critical equations.

As discussed in \cite{KMV}, the symplectic structure pertinent to the phase space $L^2_+(\R)$ is complicated. Nevertheless, $H_\pm(q)$ do generate the \eqref{CCM} flows and the mass $M(q)$ generates phase rotations.

The models \eqref{CCM} admit Lax pair formulations \cite{GL, KLV:CCM}.  Specifically, following the convention of \cite{KLV:CCM} we have that
\begin{equation}\label{LP0}
\text{$q(t)$ solves \eqref{CCM}} \quad \iff \quad \tfrac{d}{dt} L_{q(t)} = \bigl[ P_{q(t)}, L_{q(t)}\bigr],
\end{equation}
where $L_q,P_q$ are the operators acting on $L^2_+(\R)$ given by 
\begin{equation}\label{LP}
L_q = -i\partial \mp qC_+\overline{q}  \qtq{and} P_q = i\partial^2 \pm 2q\partial C_+\overline{q}.
\end{equation}

Recently there has been a great deal of progress on \eqref{CCM}, much of which has relied on complete integrability in an essential way.  First, the work \cite{KLV:CCM} established the following critical global well-posedness result (see also \cite{Badreddine:GWP} for the analogous result in the periodic setting):
\begin{theorem}[Global well-posedness, \cite{KLV:CCM}]\label{T:GWP}  \eqref{CCM} are globally well-posed in
\begin{align*}
B_{M} = \{ q\in L^2_+(\R) : \|q\|_{L^2}^2 < M_* \} \qtq{where} M_*=\begin{cases} \infty &\quad\text{in the defocusing case,}\\
2\pi &\quad\text{in the focusing case}.\end{cases}
\end{align*} 
\end{theorem}
In the focusing case, the threshold $2\pi$ which represents the mass of the ground state 
\[
Q(x)=\tfrac{\sqrt{2}}{x+i},
\]
is sharp: the work \cite{HK} showed that solutions with mass arbitrarily larger than $2\pi$ may exhibit frequency cascades, while  \cite{KKK:blowup} demonstrated that initial data with mass arbitrarily larger than $2\pi$ can blow up in finite time.  The proof of Theorem~\ref{T:GWP} in \cite{KLV:CCM} was  based in part on the derivation of a remarkable {explicit formula} for solutions to \eqref{CCM}, of the type previously established for the Benjamin--Ono and cubic Szeg\H{o} equations \cite{G:EFBO, G:EFS1, G:EFS2}:

\begin{theorem}[Explicit formula, \cite{KLV:CCM}]\label{T:EF}
Let $q(t)$ denote the global solution to \eqref{CCM} with initial data $q(0)=q\in L^2_+(\R)$, where $\|q\|_{L^2}^2 < 2\pi$ in the focusing case. Then
\begin{equation}\label{EF}
q(t,z) = \tfrac{1}{2\pi i} I_+ \big\{ \big( X + 2tL_{q} - z \big)^{-1} q \big\}\qtq{for}z\in \C\qtq{with}\Im z>0.
\end{equation}
Here $q(t,z)$ is defined via harmonic extension {\upshape(}see \eqref{PIF 2}{\upshape)}, while the operators $I_+$ and $X$ are introduced in Definitions~\ref{D:I}~and~\ref{D:X}.
\end{theorem}

In\cite{Chen},  X. Chen applied the explicit formula for defocusing \eqref{CCM} to obtain the following scattering result:

\begin{theorem}[Scattering in the defocusing case, \cite{Chen}]\label{T:Chen} Let $q(t)$ denote the global solution to the defocusing \eqref{CCM} with initial data $q\in L_+^2(\R)$. Then \begin{equation}\label{scattering}
\exists \, q^\pm\in L^2_+(\R)\qtq{such that} \lim_{t\to\pm\infty} \|q(t)- e^{it\Delta}q^\pm\|_{L^2(\R)} = 0.
\end{equation}
\end{theorem}

Furthermore, \cite{Chen} characterized the asymptotic states $q^\pm$ in terms of the initial data $q$ and the distorted Fourier transform associated to $L_{q}$.

In recent work, S. Hadama \cite{Hadama}  established small-mass scattering for \eqref{CCM} in both the focusing and defocusing cases:

\begin{theorem}[Small-mass scattering, \cite{Hadama}] Let $q(t)$ denote the global solution to \eqref{CCM} with initial data $q\in L_+^2(\R)$. If $\|q\|_{L^2}$ is sufficiently small, then \eqref{scattering} holds.
\end{theorem}
In contrast to the other works discussed above, the techniques in \cite{Hadama} do not rely on the complete integrability of \eqref{CCM}.  In fact, \cite{Hadama} treats a more general model that includes both \eqref{CCM} and the intermediate NLS equation as special cases, relying on a linear theory for Schr\"odinger equations with rough time-dependent potentials, as well as a bilinear Strichartz estimate from \cite{OT}. There is no evidence that the general model considered in \cite{Hadama} is completely integrable.

The contribution of this note is to present elementary arguments to obtain small-mass dispersive decay and scattering from the explicit formula \eqref{EF}, under suitable decay assumptions on the initial data.  We prove the following theorem:

\begin{theorem}[Pointwise decay and scattering]\label{T} Let $q(t)$ denote the global solution to \eqref{CCM} with initial data $q\in L_+^2(\R)$. If $\|q\|_{L^2}$ is sufficiently small, then the following hold: 
\begin{itemize}
\item[(i)] If $q\in L^1(\R)$, then for almost every $t\in\R$, 
\begin{equation}\label{decay}
\|q(t)\|_{L^\infty_x}\lesssim |t|^{-\frac12}(1+\|q\|_{L^2}^2)\|q\|_{L^1}.
\end{equation}
\item[(ii)] If $xq\in L^2(\R)$, then \eqref{decay} and \eqref{scattering} hold. Indeed,
\begin{align}\label{rate}
\|q(t)- e^{it\Delta}q^\pm\|_{L^2_x}\lesssim  |t|^{-1} \quad \text{as $t\to\pm\infty$}.
\end{align}
\end{itemize}
\end{theorem}

The proof of Theorem~\ref{T} relies essentially on the explicit formula \eqref{EF}.  We note, however, that the mere existence of an explicit formula does not guarantee dispersive decay as in Theorem~\ref{T}(i).    Indeed, while the Benjamin--Ono equation enjoys a similar explicit formula \cite{G:EFBO, KLV:BO}, dispersive decay may fail in this model even for arbitrarily small initial data $q\in L^1(\R)\cap L^2(\R)$. Concretely, the Lax operator $\mathcal L_q=-i\partial- C_+q$ with $q$ equal to an arbitrarily small multiple of the Benjamin--Ono soliton has negative spectrum, which is incompatible with dispersive decay; see \cite{PelinovskySulem}.

While Theorem~\ref{T}(ii) places stronger requirements on the initial data than \cite{Chen, Hadama}, we do obtain the reward of a quantitative rate of convergence. Simple time translation arguments show that one cannot hope to obtain a scattering rate under merely $L^2(\R)$ assumptions on the initial data. Moreover, we contend that the proof presented here is both straightforward and connects to existing techniques for studying long-time behavior for nonlinear Schr\"odinger equations (NLS) and related models.

To prove Theorem~\ref{T}(ii) we introduce the operator $J_{q}(t):=X-2tL_{q}$, an analogue of the Galilean operator $J_0(t)=x+2it\partial$ that is ubiquitous in the study of long-time behavior for NLS. Just as the $L^2$-norm of $J_0(t)q(t)$ is conserved for solutions to the linear Schr\"odinger equation, we will see that the $L^2$-norm of $J_{q(t)}(t)q(t)$ is conserved for solutions to \eqref{CCM}.  Moreover, when \eqref{decay} holds, the difference between $J_{q(t)}(t)$ and $J_0(t)$ is a bounded operator on $L^2$.  Thus, under the assumptions of Theorem~\ref{T} we can obtain boundedness of the $L^2$-norm of $J_0(t)q(t)$, where $q(t)$ is the solution to \eqref{CCM}.  

With such input, one can readily adapt arguments from works such as \cite{HayashiNaumkin}, which utilize the Dollard factorization for the free propagator and ODE-type methods to obtain the existence of limits as $t\to\pm\infty$. Our analysis also shows how a phase cancellation in the nonlinearity of \eqref{CCM} yields \emph{unmodified} scattering. This is in contrast to the modified scattering with logarithmic phase correction typical of cubic NLS equations (cf.  \cite{DeiftZhou, HayashiNaumkin, IfrimTataru, KatoPusateri, LindbladSogge, Murphy:scat}).

Let us finally remark that the use of an {exactly conserved} Galilean-type operator to establish dispersive decay in the completely integrable setting has precedent in the work of Ifrim and Tataru \cite{IfrimTataruBO}.  In that work, the authors study the Benjamin--Ono equation and establish long-time dispersive decay estimates for small data in the weighted space.  A key ingredient in \cite{IfrimTataruBO} is the introduction of a nonlinear analogue of the Galilean operator that is exactly conserved by the Benjamin--Ono flow, which is found by adding a nonlinear correction to the Galilean operator adapted to the underlying linear equation (in the spirit of a normal form transformation).  The authors then compare this operator to the free Galilean operator to obtain long-time control of solutions, an approach that is paralleled in the proof of Theorem~\ref{T}(ii) outlined above.

The rest of this note is organized as follows.  In Section~\ref{S:notation} we introduce notation and recall several results from \cite{KLV:CCM}. We also prove Proposition~\ref{P:A bdd}, an improved small-mass resolvent estimate for the operator appearing in the explicit formula \eqref{EF}. Finally, in Section~\ref{S:Proof} we prove our main result, Theorem~\ref{T}.

\subsection*{Acknowledgements} Part of this work was completed while the authors were participating in the program `Dispersive Integrable Equations: Pathfinders in Hamiltonian PDE' at the Institut Henri Poincar\'e. R.~K. was supported by NSF grant DMS-2452346 and the project ANR-25-CFFS-0004 ``PhysMathEDPInteg'' of the France 2030 program. T.~L. was supported by an AMS-Simons Travel Grant.  J.~M. was supported by Simons Foundation grant MPS-TSM-00006622. M.~V. was supported by NSF grant DMS-2348018.

\section{Notation and preliminaries}\label{S:notation}


The Hardy space $ L^p_+(\R)$ is the (closed) subspace of $L^p(\R)$ comprised of those functions $f$ whose Poisson integral
\begin{align}\label{PIF 2}
    f (z) :=  \int \frac{\Im z}{\pi |x-z|^2} f(x) \,dx, \qtq{defined for} \Im z >0, 
\end{align}
is holomorphic (in the upper half-plane).  By H\"older's inequality,
\begin{align}\label{PIF}
\bigl| f (z) \bigr| \leq \tfrac1\pi \bigl(\tfrac1{\Im z}\bigr)^{\frac1p} \bigl(\tfrac{\sqrt{\pi}\Gamma(p'-\frac12)}{\Gamma(p')}\bigr)^{\frac1{p'}}\| f \|_{L^p}   
\end{align}
for any $1\leq p \leq \infty$, where $p'$ denotes the H\"older dual of $p$. 

Our convention for the Fourier transform $\mathcal{F}$ is
\begin{equation*}
\hat f(\xi) = \tfrac{1}{\sqrt{2\pi}} \int_\R e^{-i\xi x} f(x)\,dx,  \qtq{so that } f(x) = \tfrac{1}{\sqrt{2\pi}} \int_\R e^{i\xi x} \hat f(\xi)\,d\xi.
\end{equation*}
With this convention, we have the Plancherel identity
\begin{equation*}
\|f\|_{L^2}=\|\hat f\|_{L^2}.
\end{equation*}

The linear Schr\"odinger propagator is defined as the Fourier multiplier operator $e^{it\Delta}=\F^{-1} e^{-it\xi^2} \F$.  This operator admits the following explicit representation:
\begin{equation}\label{schrod kernel} 
[e^{it\Delta}f](x) = (4\pi i t)^{-\frac12} \int_\R e^{\frac{i(x-y)^2}{4t}} f(y)\,dy.
\end{equation}
Using \eqref{schrod kernel}, we can directly obtain the standard dispersive estimate
\begin{equation}\label{dispersive-estimate}
\|e^{it\Delta} f\|_{L_x^\infty} \lesssim |t|^{-\frac12}\|f\|_{L^1}
\end{equation}
as well as the Dollard factorization
\begin{equation}\label{MDFM}
e^{it\Delta} = \M(t) \D(t) \F \M(t),
\end{equation}
where $\M(t) :f \mapsto e^{\frac{ix^2}{4t}}f(x)$ is a multiplication operator and $\D(t)$ denotes the $L^2$-preserving dilation
\[
[\D(t) f](x) = (2it)^{-\frac12}f(\tfrac{x}{2t}). 
\]


The definitions of the operators $I_+$ and $X$ appearing in the explicit formula \eqref{EF} are as follows:

\begin{definition}\label{D:I}
Let $I_+$ denote the (unbounded) operator on $L^2_+(\R)$ defined via
\begin{equation}\label{I+ def}
I_+(f) := \lim_{y\to\infty} \sqrt{2\pi} \int_0^\infty y e^{-y\xi} \, \wh f(\xi) \,d\xi =\lim_{\xi\downarrow 0} \sqrt{{2\pi}} \hat f(\xi),
\end{equation}
with domain $D(I_+)$ given by the set of those $f\in L^2_+(\R)$ for which this limit exists.
\end{definition}

\begin{definition}\label{D:X}
Let $X$ denote the (unbounded) operator on $L^2_+(\R)$ defined as the generator of the contraction semigroup
\begin{equation}\label{E:D:X}
e^{-itX} f = \tfrac{1}{\sqrt{2\pi}} \int_0^\infty e^{i\xi x} \widehat f(\xi + t)\,d\xi = C_+\bigl( e^{-itx} f \bigr)
\end{equation}
acting on $L^2_+(\R)$.  Concretely,
\[
D(X) = \bigl\{ f\in L^2_+ (\R): \widehat f\in H^1\bigl([0,\infty)\bigr) \bigr\} \quad \text{and} \quad \widehat{Xf}(\xi) =i\tfrac{d{\widehat f}}{d\xi}(\xi)\quad \text{for} \quad f\in D(X).
\]
\end{definition}

\begin{remark}\label{Xx}
If $f\in L^2_+(\R)$ satisfies $\langle x\rangle f(x)\in L^2(\R)$, then $f\in D(X)$ and $[Xf](x)=xf(x)$.   
\end{remark}

The spectrum of $X$ consists of the closed lower half-plane.  For $\Im z>0$, the resolvent of $X$ is given by
\begin{equation}\label{X}
(X-z)^{-1} f = \tfrac{f(x)-f(z)}{x-z}.
\end{equation}

The following proposition recapitulates results from \cite{KLV:CCM}. 

\begin{proposition}
For $t\in \R$ and $z\in \C$ with $\Im z>0$, the operator
\begin{equation}\label{A0}
A_0(t,z) := (X +2tL_0 -z)^{-1} = e^{-it\Delta} (X - z)^{-1}e^{it\Delta} 
\end{equation}
satisfies the following estimates:
\begin{align}
\norm{A_0(t,z)}_{L^2_+\to L^2_+} &\leq (\Im z)^{-1},\label{A0 1} \\
\sup_{t\neq 0} |t|\norm{A_0(t,z) }_{L^1_+\to L^\infty_+} &\lesssim 1.\label{A0 3}
\end{align}
The range of $A_0$ lies in the domain of $I_+$,
\begin{equation}\label{A0 4} 
I_+\bigl( A_0(t,z) f \bigr)  = 2\pi i \big[ e^{it\Delta} f \big](z) \qtq{for all} t\in\R\qtq{and} \Im z>0,
\end{equation}
and the composition $I_+\circ A_0$ is bounded:
\begin{align}
\big| I_+[A_0(t,z)f]| &\lesssim \min\big\{ (\Im z)^{-\frac{1}{2}} \norm{f}_{L^2}, |t|^{-\frac{1}{2}} \norm{f}_{L^1} \big\} .\label{A0 5}
\end{align}
\end{proposition}

For $q\in L^2_+(\R)$, $t\in \R$, and $z\in \C$ with $\Im z>0$, the operator
\begin{equation*}
A(t,z;q) := (X + 2tL_q - z)^{-1}
\end{equation*}
appearing in the explicit formula \eqref{EF} was constructed in \cite[Proposition 4.10]{KLV:CCM} as the resolvent of a maximally accretive operator, which coincides with the naive notion of the sum $X+2t L_q$ when restricted to  $D(X)\cap D(L_0)$.

Mapping properties of the operator $A(t,z;q)$ were further elaborated in \cite[Proposition 4.11]{KLV:CCM}.  In particular, it was shown that for fixed $T>0$, 
\begin{align}\label{not enough}
|t|\norm{ A(t,z;q) }_{L^1_+ \to L^\infty_+} \lesssim 1+ (\Im z)^{-1},
\end{align}
uniformly for $\Im z>0$, $|t|\leq T$, and $q$ belonging to a fixed bounded and equicontinuous subset of $L^2_+(\R)$. To prove the dispersive decay in Theorem~\ref{T}, we strengthen \eqref{not enough} in two essential ways: we derive an upper bound on the $L^1_+ \to L^\infty_+$ norm of $A(t,z;q)$ that (i) does not deteriorate as $z$ approaches the real axis and (ii) is uniform for $t\neq 0$.  It is to achieve these two goals that we impose the small mass assumption. 

\begin{proposition}\label{P:A bdd}
There exists $\delta_0>0$ so that the operator $A(t,z;q)$ satisfies
\begin{align} \label{off axis 2}
|t|\,\|A(t,z;q)\|_{L^1_+ \to L^\infty_+} \lesssim 1,
\end{align}
uniformly for $\Im z>0$, $t\neq 0$, and $q\in L^2_+(\R)$ with $\|q\|_{L^2} \leq \delta_0$.
\end{proposition}

\begin{proof}
For any $\ell\geq 0$, we may bound
\begin{align*}
\bigl\| A_0(t,z) &\bigl[\pm 2tqC_+\overline{q} A_0(t,z) \bigr]^\ell\bigr\|_{L^1_+ \to L^\infty_+}\\
&\leq \| A_0(t,z)\|_{L^1_+ \to L^\infty_+} \bigl\|2tqC_+\overline{q} A_0(t,z)\bigr\|_{L^1_+ \to L^1_+}^\ell\\
&\leq \| A_0(t,z)\|_{L^1_+ \to L^\infty_+}  \bigl[ 2t \|q\|_{L^2}^2 \|C_+\|_{L^2 \to L^2_+}\| A_0(t,z)\|_{L^1_+ \to L^\infty_+} \bigr]^\ell.
\end{align*}
Thus, by \eqref{A0 3} there exists a constant $C>0$ so that
\begin{align*}
\bigl\| A_0(t,z) \bigl(\pm 2tqC_+\overline{q} A_0(t,z) \bigr)^\ell\bigr\|_{L^1_+ \to L^\infty_+}\leq C|t|^{-1}  \bigl[ 2C\delta_0^2 \bigr]^\ell,
\end{align*}
uniformly for $\Im z>0$, $t\neq 0$, and $q\in L^2_+(\R)$ with $\|q\|_{L^2} \leq \delta_0$.

Choosing $\delta_0$ sufficiently small guarantees the convergence of the series 
\[
\sum_{\ell\geq0}A_0(t,z) \bigl[\pm 2tqC_+\overline{q} A_0(t,z) \bigr]^\ell
\]
in the $L^1_+(\R)\to L^\infty_+(\R)$ operator norm.  As this series represents the Born series of $A(t,z;q)$, we obtain
\[
\|A(t,z;q)\|_{L^1_+ \to L^\infty_+}\leq \sum_{\ell\geq 0} C|t|^{-1}  \bigl[ 2C\delta_0^2 \bigr]^\ell \leq \tfrac{C}{1-2C\delta_0^2}|t|^{-1} \qtq{whenever} 2C\delta_0^2<1,
\]
uniformly for $\Im z>0$, $t\neq 0$, and $q\in L^2_+(\R)$ with $\|q\|_{L^2} \leq \delta_0$.
\end{proof}

\begin{remark}
The constant $\delta_0$ in Proposition~\ref{P:A bdd} depends solely on the implicit constant in \eqref{A0 3}.  It is not difficult to find an upper bound for this constant using \eqref{X}, \eqref{A0}, and a stationary phase analysis.  However, as this argument leads to a constant with no physical significance, we do not pursue it here.
\end{remark}

We next recapitulate \cite[Proposition~2.3]{KLV:CCM}, which constructs the unitary group generating solutions to \eqref{CCM}. 

\begin{proposition}\label{P:unique} Let $q(t)$ be a global $H^\infty_+(\R)$ solution to \eqref{CCM}. For all $\psi_0\in L_+^2(\R)$, the initial value problem
\[
\tfrac{d}{dt} \psi(t) =P_{q(t)}\psi(t),\quad \psi(0)=\psi_0
\]
admits a unique $C_t L_+^2\cap C_t^1 H_+^{-2}$ solution, which is global in time. For each $t\in \R$, the mapping $U(t):\psi_0\to\psi(t)$ is unitary on $L_+^2(\R)$,
\[
q(t) = U(t)q(0),\qtq{and} L_{q(t)} = U(t)L_{q(0)} U(t)^*. 
\]
If $\psi_0\in H_+^\infty(\R)$, then so is $\psi(t)$ for all $t\in\R$. If $\psi_0\in H_+^\infty(\R)$, $\langle x\rangle\psi_0\in L^2(\R)$, and $\langle x\rangle q(0)\in L^2(\R)$, then $\langle x\rangle \psi(t)\in L^2(\R)$ for all $t\in\R$. 
\end{proposition} 

We will employ the following commutator identity; see \cite[Lemma~4.6]{KLV:CCM}:

\begin{lemma}\label{X commute} If $q\in H_+^\infty(\R)$ and $\langle x\rangle q\in L^2(\R)$, then
\begin{equation}\label{commutator}
[X,P_q]= 2L_q. 
\end{equation}
\end{lemma}

\section{Dispersive decay and scattering}\label{S:Proof} 
This section is devoted to the proof of Theorem~\ref{T}.  We treat each part separately.

\begin{proof}[Proof of Theorem~\ref{T}(i)] Suppose the initial data $q\in L_+^2(\R)$ satisfies $\|q\|_{L^2}\leq \delta_0$, where $\delta_0$ is the small parameter in Proposition~\ref{P:A bdd}.  Suppose further that $q\in L^1(\R)$, and let $q(t)$ denote the corresponding global solution to \eqref{CCM}.

Recall that any function $f\in L^2_+(\R)$ admits a holomorphic extension to the upper half-plane:
\begin{align*}
f(x+ib) =  \int \frac{b}{\pi |y-(x+ib)|^2} f(y) \,dy \qtq{defined for} b >0;
\end{align*}
see \eqref{PIF 2}.  Moreover, for almost every $x\in\R$, the boundary value $f(x)$ can be obtained as the pointwise limit of $f(x+ib)$ as $b\downarrow 0$.
In this way, to derive the dispersive estimate \eqref{decay}, it suffices to obtain pointwise time decay for the solution $q(t,x+ib)$ that is independent of the height $b>0$ in the upper half-plane.  We will analyze $q(t,x+ib)$ via the explicit formula \eqref{EF}.


Using \eqref{EF} followed by the resolvent identity, we write
\begin{align*}
q(t,x+ib)
&=\tfrac{1}{2\pi i} I_+\bigl[A(t,x+ib ;q) q\bigr]\\
&= \tfrac{1}{2\pi i} I_+ \bigl[A_0(t,x+ib)q\bigr] \pm \tfrac{1}{2\pi i} I_+ \bigl[ A_0(t,x+ib) 2t qC_+\ol{q} A(t,x+ib ;q) q \bigr] .
\end{align*}

In view of \eqref{A0 4},
\begin{equation*}
\tfrac{1}{2\pi i}I_+ \bigl[A_0(t,x+ib)q\bigr]= [e^{it\Delta} q] (x+ib),
\end{equation*}
and using \eqref{PIF} and \eqref{dispersive-estimate} we may bound
\begin{equation*}
\bnorm{ [e^{it\Delta} q] (x+ib) }_{L^\infty_x}
\leq \norm{ e^{it\Delta} q }_{L^\infty_x}
\lesssim |t|^{-\frac{1}{2}} \bnorm{q}_{L^1} .
\end{equation*}

Arguing similarly, we write
\begin{equation*}
\tfrac{1}{2\pi i} I_+ \bigl[ A_0(t,x+ib) 2t qC_+\ol{q} A(t,x+ib ;q) q \bigr]
= \bigl[e^{it\Delta} \bigl(2t qC_+\ol{q} A(t,x+ib ;q) q\bigr) \bigr] (x+ib)
\end{equation*}
and use \eqref{PIF}, \eqref{dispersive-estimate}, and \eqref{off axis 2} to estimate
\begin{align*}
\bnorm{  \bigl[e^{it\Delta} \bigl(2t qC_+\ol{q} A(t,x+ib ;q) &q\bigr) \bigr] (x+ib) }_{L^\infty_x} \\
&\leq \bnorm{ e^{it\Delta} \bigl(2t C_+q A(t,x+ib ;q) q\bigr) }_{L^\infty_x} \\
&\lesssim |t|^{-\frac{1}{2}} \bnorm{ 2t qC_+\ol{q} A(t,x+ib ;q) q }_{L^1_x} \\
&\lesssim |t|^{\frac{1}{2}} \snorm{q}_{L^2}^2 \bnorm{ A(t,x+ib ;q)}_{L^1_+\to L^\infty_+}\|q\|_{L^1} \\
&\lesssim |t|^{-\frac{1}{2}} \|q\|_{L^2}^2 \snorm{q}_{L^1} .
\end{align*}

Combining these bounds we arrive at
\begin{equation*}
\norm{ q(t,x+ib) }_{L^\infty_x} \lesssim (1+\|q\|_{L^2}^2) |t|^{-\frac{1}{2}} \snorm{q}_{L^1}
\end{equation*}
uniformly for $t\neq 0$ and $b>0$.  Taking the limit $b\downarrow 0$ yields \eqref{decay}.\end{proof}

\begin{proof}[Proof of Theorem~\ref{T}(ii)] As above, we suppose the initial data $q\in L_+^2(\R)$ satisfies $\|q\|_{L^2}\leq\delta_0$. Suppose further that $xq\in L^2(\R)$, and let $q(t)$ denote the corresponding global solution to \eqref{CCM}.  By Cauchy--Schwarz, we have $q\in L^1(\R)$, so that Theorem~\ref{T}(i) guarantees the dispersive decay estimate \eqref{decay} for $q(t)$.  

We now introduce a sequence $\{q_n\}_{n=1}^\infty\subset H_+^\infty(\R)$ such that $\langle x\rangle q_n \in L^2(\R)$, with $q_n\to q$ in $L^2(\R)$ and $x q_n\to xq$ in $L^2(\R)$.  In particular, we may assume \begin{align}\label{hyp}
\|xq_n\|_{L^2}\lesssim1 \qtq{and} \|q_n\|_{L^2}\leq \delta_0 \quad \text{for all $n\geq 1$}.
\end{align}

Let $q_n(t)$ denote the solutions to \eqref{CCM} with initial data $q_n$.  In view of \eqref{hyp}, these satisfy
\begin{align}\label{1.5 for qn}
\sup_{t\in \R} |t|^{\frac12} \|q_n(t)\|_{L^\infty_x} \lesssim (1 + \|q_n\|_{L^2}^2) \|q_n\|_{L^1}\lesssim 1,
\end{align}
with the implicit constant uniform in $n\geq 1$. By $L^2$ well-posedness, we have that $q_n(t)\to q(t)$ in $L^2$ for any $t\in\R$. 

Writing $U_n(t)$ for the unitary group associated to $q_n(t)$ constructed in Proposition~\ref{P:unique}, we begin by establishing
\begin{equation}\label{Jq}
U_n(t)X\psi = (X-2tL_{q_n(t)})U_n(t)\psi
\end{equation}
for all $\psi\in H_+^\infty(\R)$ satisfying $\langle x\rangle\psi\in L^2(\R)$. This identity holds trivially at $t=0$, so that by the uniqueness asserted in Proposition~\ref{P:unique}, it suffices to show
\[
\tfrac{d}{dt}[U_n(t)X - (X-2t{L}_{q_n(t)})U_n(t)]\psi = {P}_{q_n(t)}[U_n(t)X - (X-2t{L}_{q_n(t)} )U_n(t)]\psi.
\]
To this end, we compute directly using Proposition~\ref{P:unique}, \eqref{LP0}, and Lemma~\ref{X commute}:
\begin{align*}
\tfrac{d}{dt}&[U_n(t) X - (X-2t{L}_{q_n(t)})U_n(t)]\psi \\
& = {P}_{q_n(t)} U_n(t) X\psi - (X-2t{L}_{q_n(t)}){P}_{q_n(t)} U_n(t) \psi \\
& \quad + 2{L}_{q_n(t)}U_n(t)\psi + 2t[{P}_{q_n(t)},{L}_{q_n(t)}]U_n(t)\psi \\
& = {P}_{q_n(t)}[U_n(t) X - (X-2t{L}_{q_n(t)})U_n(t)]\psi  - [X,{P}_{q_n(t)}]U_n(t)\psi + 2{L}_{q_n(t)} U_n(t)\psi \\
& =  {P}_{q_n(t)}[U_n(t) X - (X-2t{L}_{q_n(t)})U_n(t)]\psi.
\end{align*}
Claim \eqref{Jq} now follows from Proposition~\ref{P:unique}, as explained above.

Using \eqref{Jq}, $q_n(t)=U_n(t)q_n$, and the unitarity of $U_n(t)$, we now obtain
\begin{align*}
\|(X-2t{L}_{q_n(t)})q_n(t)\|_{L^2_x} = \|U_n(t)Xq_n\|_{L^2_x} = \|xq_n\|_{L^2} \lesssim 1
\end{align*}
uniformly for $t\in\R$ and $n\geq 1$. By \eqref{hyp} and \eqref{1.5 for qn}, we have 
\[
|t|\, \| q_n(t) C_+(|q_n(t)|^2)\|_{L^2_x} \lesssim |t| \,\|q_n(t)\|_{L^\infty_x} \| |q_n(t)|^2\|_{L^2_x} \lesssim |t|\, \|q_n(t)\|_{L^\infty_x}^2 \|q_n(t)\|_{L^2_x} \lesssim 1
\]
uniformly for $t\in\R$ and $n\geq 1$.  Combining the last two displays, we obtain
\begin{equation}\label{Jbd}
\begin{aligned}
&\|(x+2it\partial)q_n(t) \|_{L^2_x} \\
&\qquad \qquad \lesssim \|(X-2t{L}_{q_n(t)})q_n(t)\|_{L^2_x} + |t|\,\| q_n(t) C_+(|q_n(t)|^2)\|_{L^2_x} \lesssim 1
\end{aligned}
\end{equation}
uniformly for $t\in\R$ and $n\geq 1$. 

We turn now to the proof of scattering.  We define
\[
f_n(t) = e^{-it\Delta} q_n(t) \ \ \ \text{and}\ \ \ g_n(t) = \F \M(t) f_n(t), \ \ \ \text{so that}\ \ \ q_n(t) = \M(t) \D(t) g_n(t). 
\]
By direct calculation using \eqref{CCM}, \eqref{MDFM}, and the fact that $C_+$ commutes with dilations, we find 
\begin{align*}
i\partial_t \hat f_n(t)&  = \F e^{-it\Delta}[\pm 2i q_n(t) C_+(|q_n(t)|^2)']\\
&=(4t|t|)^{-1} \F \overline{\M(t)} \F^{-1}[ \pm 2i g_n(t) C_+(|g_n(t)|^2)']. 
\end{align*}
Thus, by Plancherel,
\begin{align}\label{deriv fn}
\|\partial_t \hat f_n(t)\|_{L^2} &\lesssim |t|^{-2} \|g_n(t) C_+(|g_n(t)|^2)'\|_{L^2_x} \lesssim |t|^{-2} \|g_n(t)\|_{L^\infty_x}^2 \|  g'_n(t)\|_{L^2_x}. 
\end{align}

Employing Cauchy--Schwarz, the identity $x+2it\partial = e^{it\Delta} x e^{-it\Delta}$, \eqref{hyp}, and \eqref{Jbd}, we can bound
\begin{align*}
\|g_n(t)\|_{L^\infty_x} \lesssim \|f_n(t)\|_{L^1_x} & \lesssim \|f_n(t)\|_{L^2_x} + \|xf_n(t)\|_{L^2_x}\\
&  \lesssim \| q_n(t)\|_{L^2_x} + \|(x+2it\partial) q_n(t)\|_{L^2_x} \lesssim 1
\end{align*}
uniformly for $t\in\R$ and $n\geq 1$.  Using also Plancherel, we obtain
\[
\|g_n'(t)\|_{L^2_x} = \|xf_n(t)\|_{L^2_x} = \|(x+2it\partial)q_n(t)\|_{L^2_x} \lesssim 1
\]
uniformly for $t\in\R$ and $n\geq 1$. Inserting the last two bounds into \eqref{deriv fn}, we find
\begin{align*}
\|\partial_t \hat f_n(t)\|_{L^2_x} \lesssim |t|^{-2} \quad\text{uniformly for $|t|\geq 1$ and $n\geq 1$. }
\end{align*}

By the Plancherel Theorem and the Fundamental Theorem of Calculus, we therefore deduce that for each $n\geq 1$ there exists $q_n^+\in L_+^2$ such that
\begin{align}\label{1210}
\|f_n(t) - q_n^+\|_{L^2_x} \lesssim t^{-1}\qtq{uniformly for}t\geq 1\qtq{and}n\geq 1.
\end{align}
Moreover, for any $t\geq T\geq 1$ and $n,m\geq 1$, we have
\begin{align*}
\|f_n(t)-f_m(t)\|_{L^2_x} & \leq \|f_n(T)- f_m(T)\|_{L^2} + \int_T^t \|\partial_s \hat f_n(s)\|_{L^2_x} + \|\partial_s \hat f_m(s)\|_{L^2_x} \,ds \\
& \lesssim \|q_n(T)-q_m(T)\|_{L^2} + T^{-1}.
\end{align*}
Using $L^2$ well-posedness, we can therefore derive that $\{q_n^+\}_{n\geq 1}$ is Cauchy in $L^2(\R)$ and hence converges to a limit $q^+\in L_+^2(\R)$ as $n\to\infty$.  

Using \eqref{1210} and the triangle inequality, we have
\begin{align*}
\|e^{-it\Delta} q(t) - q^+ \|_{L^2_x} & \leq \| e^{-it\Delta} q(t) - f_n(t) \|_{L^2_x} + \|f_n(t) - q_n^+\|_{L^2_x} + \|q_n^+ - q^+ \|_{L^2} \\
& \lesssim \| q(t)-q_n(t)\|_{L^2_x} + \|q_n^+- q^+\|_{L^2} + t^{-1}
\end{align*} 
for any $t,n\geq 1$. Using $L^2$ well-posedness once again, we conclude that $e^{-it\Delta}q(t) \to q^+$ in $L^2(\R)$ as $t\to+\infty$ and moreover that \eqref{rate} holds as $t\to +\infty$.  A parallel argument covers the limit $t\to -\infty$, completing the proof of Theorem~\ref{T}. \end{proof}

\end{document}